\documentclass[11pt]{article}
\usepackage{amssymb}
\usepackage{amsmath}
\usepackage{graphicx}
\usepackage{float}
\usepackage{setspace}

\newtheorem{theorem}{Theorem}[section]
\newtheorem{algorithm}[theorem]{Algorithm}

\newtheorem{corollary}[theorem]{Corollary}
\newtheorem{lemma}[theorem]{Lemma}
\newtheorem{claim}[theorem]{Claim}
\newtheorem{proposition}[theorem]{Proposition}

\newtheorem{definition}[theorem]{Definition}
\newtheorem{question}[theorem]{Question}

\newcommand*{\qed}{\hfill\ensuremath{\blacksquare}}%
\newcommand{\ex}{{\mathrm{ex}}}

\def\endproofbox{\hskip 1.3em\hfill\rule{6pt}{6pt}}
\newenvironment{proof}%
{%
\noindent{\it Proof.}
}%
{%
 \quad\hfill\endproofbox\vspace*{2ex}
}
\def\qed{\hskip 1.3em\hfill\rule{6pt}{6pt}}

\def\e{\varepsilon}
\def\ex{{\rm ex}}

\def\cC{{\cal C}}
    
\def\cF{{\cal F}}

\def\ex{{\rm ex}}

\def\1e{\frac{1}{\e}\log \frac{1}{\e}}

\def \e{\epsilon}

\begin{document}

\title{\huge\bf  Linear cycles of consecutive lengths}

\author{
Tao Jiang\thanks{Department of Mathematics, Miami University, Oxford,
OH 45056, USA. E-mail: jiangt@miamioh.edu. Research supported in part
by National Science Foundation grant DMS-1855542. }
\quad \quad Jie Ma \thanks{
School of Mathematical Sciences,
University of Science and Technology of China, Hefei, 230026,
P.R. China. Email: jiema@ustc.edu.cn.
Research supported in part by NSFC grant 11622110.
}
\quad \quad Liana Yepremyan
\thanks{
Department of Mathematics, Statistics, and Computer Science, 
University of Illinois at Chicago, IL 60607, USA, and Department of
Mathematics, London School of Economics, London WC2A 2AE, UK,
l.yepremyan@lse.ac.uk, lyepre2@uic.edu, Research supported by
Marie Sklodowska Curie Global Fellowship, H2020-MSCA-IF-2018:846304.
} }

\date{June 22, 2020}

\maketitle

\begin{abstract}
A well-known result of Verstra\"ete \cite{V00} shows that for each integer $k\geq 2$ every graph $G$ with
average degree at least $8k$ contains cycles of $k$ consecutive even lengths, the shortest of which
is at most twice the radius of $G$. We establish two extensions of Verstra\"ete's result for linear cycles in linear $r$-uniform hypergraphs.

We show that for any fixed integers $r\geq 3,k\geq 2$, there exist constants $c_1=c_1(r)$ and $c_2=c_2(r,k)$, such that every linear $r$-uniform  hypergraph $G$ with average degree  $d(G)\geq c_1 k$ contains linear cycles of $k$ consecutive even lengths, the shortest of which is at most $2\lceil \frac{ \log n}{\log (d(G)/k)-c_2}\rceil$. 
In particular, as an immediate corollary, we retrieve the current best known upper bound 
on the linear Tur\'an number of $C^r_{2k}$ with improved coefficients.

Furthermore, we show that for any fixed integers $r\geq 3,k\geq 2$, there exist constants $c_3=c_3(r)$
and $c_4=c_4(r)$ such that every $n$-vertex linear $r$-uniform graph with average degree $d(G)\geq c_3k$, contains
linear cycles of $k$ consecutive lengths, the shortest of which has length at most 
$6\lceil \frac{\log n}{\log (d(G)/k)-c_4} \rceil +6$.
Both the degree condition and the shortest length among the cycles guaranteed are best possible up to a constant factor. 
\end{abstract}

\section{Introduction}
For $r\geq 3$, an $r$-uniform hypergraph (henceforth, $r$-graph) is {\it linear} if any two edges share at most one vertex.
An $r$-uniform {\it linear cycle} of length $k$, denoted by $C_k^{r}$, is a linear $r$-graph consisting of $k$ edges $e_1,e_2,...,e_k$
on $(r-1)k$ vertices such that $|e_i\cap e_j|=1$ if $j=i\pm1$ (indices taken modulo $k$) and $|e_i\cap e_j|=0$ otherwise. For $r=2$, linear $r$-graphs are just the usual graphs, and so are the linear cycles. Motivated by the known results for graphs, we study sufficient conditions for the existence of linear cycles of given lengths in linear $r$-graphs for $r\geq 3$. Our results  apply to linear $r$-graphs of a broad edge density, covering both sparse and dense graphs.

\subsection{History}

The line of research about the distribution of cycle lengths in graphs was initiated by Burr and Erd\H{o}s (see \cite{Erd76}) who conjectured that for every odd number $k$, there is a
constant $c_k$ such that for every natural number $m$, every graph of average degree at least $c_k$ contains a cycle of length $m$ modulo $k$.  This conjecture was confirmed in this full generality by Bollob\'as~\cite{bollobas} for $c_k=2((k+1)^k-1)/k$, although earlier partial results were obtained by Erd\H{o}s and Burr~\cite{Erd76} and Robertson~\cite{Erd76}. The constant $c_k$ was improved to $8k$  by Verstra\"ete \cite{V00}. Thomassen~\cite{thomassen1,thomassen2} strengthened the result of Bollob\'as by proving that for every $k$
(not necessarily odd), every graph with minimum degree at least $4k(k + 1)$ contains cycles of
all even lengths modulo $k$.  

On a similar note, Bondy and Vince \cite{BV98} proved a conjecture of Erd\H{o}s in a strong form showing that any graph with minimum degree at least three contains two cycles whose lengths differ by one or two. Since then there has been  extensive research (such as \cite{HS98,Fan,SV08, Ma,LM}) on the  general problem of finding $k$ cycles of consecutive (even or odd) lengths under minimum degree or average degree conditions in  graphs. Very recently, the optimal minimum degree condition assuring the existence of such $k$ cycles was announced in \cite{GHLM}. 

The problem of finding  consecutive length cycles in $r$-graphs is related to another classical problem in extremal graph theory, namely Tur\'an numbers for cycles in graphs and hypergraphs. For $r\geq 2$, the {\it Tur\'an number} $\ex(n,\cF)$ of a family $\cF$ of $r$-graphs is the maximum number of edges in an $n$-vertex $r$-graph which does not contain any member of $\cF$ as its subgraph.  If $\cF$ consists of a single graph $F$, we write $\ex(n,F)$ for $\ex(n,\{F\})$.
A well-known result of Erd\H{o}s (unpublished) and independently of Bondy and Simonovits \cite{BS74} 
states that for any integer $k\geq 2$, there exists some absolute constant $c>0$ such that $\ex(n,C_{2k})\leq ck n^{1+1/k}$.
The value of $c$ was further improved by the results of Verstra\"ete~\cite{V00} and Pikhurko~\cite{Pik},
and the current best known upper bound is $\ex(n,C_{2k})\leq 80\sqrt{k}\log k n^{1+1/k}$, due to Bukh and Jiang \cite{BJ}. Verstra\"ete's main result from~\cite{V00} is as follows.

\begin{theorem}{\rm (Verstra\"ete, \cite{V00})}\label{thm:Ver}
Let $k\geq 2$ be an integer and $G$ a bipartite graph of average degree at least $4k$ and girth $g$.
Then there exist cycles of $(g/2-1)k$ consecutive even lengths in $G$, the shortest of which has length at most twice the radius of $G$.
\end{theorem}

In Theorem~\ref{thm:Ver}, in addition to finding $k$ cycles of consecutive even lengths we also see an upper bound on the shortest length among these cycles. Thus it immediately yields $\ex(n,C_{2k})\leq 8kn^{1+1/k}$, which
improves on the coefficients in the theorems of Erd\H{o}s and of Bondy-Simonovits. Notice that Verstra\"ete's theorem is applicable to both sparse and dense host graphs while arguments establishing bounds on $\ex(n,F)$ directly usually address  relatively dense host graphs. For example, for $F=C_{2k}$, these would typically be graphs with average degree at least $\Omega(n^{1/k})$.  

For hypergraphs, Verstra\"ete~\cite{Verstraete-survey} conjectured that for $r\geq 3$ any $r$-graph with average degree $\Omega(k^{r-1})$ contains Berge cycles of $k$ consecutive lengths where an $r$-uniform {\it Berge cycle} of length $k$ is a hypergraph containing $k$  vertices $v_1,...,v_k$ and $k$ distinct edges $e_1,...,e_k$ such that $\{v_i,v_{i+1}\}\subseteq e_i$ for each $i$, where the indices are taken modulo $k$. 
Let $\mathcal{B}^r_k$ denote the family of $r$-uniform Berge cycles of length $k$. Results of \cite{FO,G06,GL-3uniform,GL} showed that for all $k,r\geq 3$, $\ex(n,\mathcal{B}_k^r)\leq c_{k,r}\cdot n^{1+1/\lfloor k/2\rfloor}$, where $c_{k,r}=O(k^r)$.  Jiang and Ma in \cite{JM} confirmed Verstra\"ete's conjecture on the existence of Berge cycles of consecutive lengths,  and just as in Theorem~\ref{thm:Ver}, they were able to control the length of the shortest cycle in the collection which implied an improved $c_{k,r}$
by an $\Omega(k)$ factor in the upper bound of  $\ex(n,\mathcal{B}_k^r)$. As an intermediate step and a result of independent interest,
they also proved the following result.
\begin{theorem}{\rm (Jiang and Ma, \cite{JM})}\label{thm:JM}
For all $r\geq 3$, any linear $r$-graph with average degree at least $7r(k+1)$ contains Berge cycles of $k$ consecutive lengths.
\end{theorem}

Theorem~\ref{thm:JM} suggests that the problem of finding Berge cycles of consecutive lengths in general $r$-graphs bears some resemblance to the graph case, 
but that is not the case for linear cycles. Indeed, the Tur\'an number  $\ex(n, C^r_k)$ of the linear cycle $C_k^r$ was determined precisely for large $n$  by F\"uredi and Jiang \cite{FJ} for $r\geq 5$ and independently by Kostochka, Mubayi and Verstra\"ete \cite{KMV} for $r\geq 3$.
Asymptotically their results show that $\ex(n,C^r_k)\sim  \lfloor \frac{k-1}{2} \rfloor \binom{n}{r-1}$.

However, if we study the emergence of linear cycles in linear host hypergraphs instead then the behavior of the Tur\'an numbers of linear cycles
bears much more resemblance to the graph case.
To be more precise, let us define the {\it linear Tur\'an number} $\ex_L(n,H)$ of a linear $r$-graph $H$ to be the maximum number of edges in an $n$-vertex linear $r$-graph $G$ that does not contain $H$ as a subgraph.
Quoting \cite{Verstraete-survey}, the problem of determining the linear Tur\'an number of a linear cycle ``seems to be a more faithful generalization of the even cycle problem in graphs".  Indeed, Collier-Cartaino, Graber, and Jiang \cite{CGJ}  proved that
for all integers $r,k\geq 2$ there exist positive constants $c=c(r,k)$, $d=d(r,k)$ such that $\ex_L(n, C^r_{2k})\leq c n^{1+1/k}$ and 
$\ex_L(n, C^r_{2k+1})\leq d  n^{1+1/k}.$
 For fixed $r$, the constants $c(r,k)$ and $d(r,k)$ established  are exponential in $k$. As a corollary, one of the main results we prove, Theorem \ref{thm:even-lengths} implies that $c=c(r,k)$ can be
taken quadratic in $k$, improving the results in~\cite{CGJ}. Note that these results on linear Tur\'an numbers of linear even cycles can be viewed as a generalization of the Bondy-Simonovits even cycle theorem,
while the result on odd linear cycles demonstrates a phenomenon that is very different from the graph case.
To this end, note that the study of $\ex_L(n,C^3_3)$ is equivalent to the famous $(6,3)$-problem, which is
to determine the maximum number of edges $f(n,6,3)$ in an $n$-vertex $3$-graph such that no six vertices span three or more edges. Ruzsa and Szemr\'edi \cite{RS} showed that for some constant $c>0$, $n^{2-c\sqrt{\log n}} < f(n,6,3) =o(n^2)$, where the upper bound uses the regularity lemma and the lower bound uses Behrend's construction  \cite{beh} of dense subsets of $[n]$ not containing $3$-term arithmetic progressions.

\subsection{Our results}

We establish two extensions of Theorem~\ref{thm:Ver} for linear cycles in linear $r$-uniform hypergraphs.
First, we give a generalization of Theorem \ref{thm:Ver} for even linear cycles in linear $r$-graphs along with a near optimal control on the shortest length of the even cycles obtained.
\begin{theorem} \label{thm:even-lengths}
Let $r\geq 3$ and $k\geq 2$ be integers. 
Let $c_1=128r^{2r+3}$ and $c_2=\log (64kr^{2r+2})$. If $G$ is an $n $-vertex linear $r$-graph with average degree $d(G)\geq c_1 k$
then $G$ contains linear cycles of $k$ consecutive even lengths, the shortest of which is at most
$2\lceil \frac{\log n}{\log (d(G)/k)-c_2}\rceil$.
\end{theorem}
Theorem \ref{thm:even-lengths} immediately implies an improved upper bound on the linear Tur\'an number of linear even cycles which previously was $cn^{1+1/k}$ for some $c$  exponential in $k$~\cite{CGJ}, for fixed $r$.

\begin{corollary} \label{cor:linear-turan}
Let $r\geq 3, k\geq 2$ be integers. For all  $n$, \[ex_L(n, C^r_{2k})\leq 64k^2r^{2r+3} n^{1+1/k}.\]
\end{corollary}

Our next main result shows that under analogous degree conditions as in Theorem \ref{thm:even-lengths}, we can in fact ensure linear cycles of $k$ consecutive lengths (even and odd both included), not just linear cycles of $k$ consecutive even lengths. Furthermore,
the length of the shortest cycle in the collection is within a constant factor of being optimal. Note that such a phenomenon
can only exist in $r$-graphs with $r\geq 3$, as for graphs, one needs more than $n^2/4$ edges in an $n$-vertex graph just to ensure the existence of any odd cycle.

\begin{theorem} \label{thm:all-lengths}
Let $r\geq 3$ and $k\geq 1$ be integers. There exist  constants $c_1, c_2$ depending on $r$ such that
if $G$ is an $n$-vertex  linear $r$-graph with average degree $d(G)\geq c_1 k$ then $G$ contains linear cycles of $k$ consecutive lengths,
the shortest of which is at most  $ 6\lceil\frac{\log n}{\log (d(G)/k)-c_2}\rceil+6$.
\end{theorem}

When viewed as a result on the average degree needed to ensure cycles of consecutive lengths, Theorem \ref{thm:all-lengths} is a substantial strengthening of both Theorem \ref{thm:JM} and
Theorem \ref{thm:even-lengths}. However, the control on the shortest length of a cycle in the collection is weaker than those in Theorem \ref{thm:even-lengths} and in \cite{JM} by roughly a factor of $3$. 
As a result, while Theorem \ref{thm:even-lengths} yields $\ex_L(n,C^r_{2k})=O(n^{1+1/k})$, Theorem \ref{thm:all-lengths} would only give us $\ex(n,C^r_{2k+1})=O(n^{1+3/k})$, and hence
it does not imply the bound on $\ex_L(n,C^r_{2k+1})$ given in~\cite{CGJ}.

Finally, note that the  shortest lengths of  linear cycles that we find in Theorem \ref{thm:even-lengths} and Theorem \ref{thm:all-lengths} are within a constant factor of being optimal, due to the following proposition which can be proved using a standard deletion argument. We delay its proof to the appendix.

\begin{proposition}  \label{prop:optimal-lengths}
Let $r\geq 2$ be an integer. For every real $\epsilon>0$ there exists a positive integer $n_0$ such that for all integers $n\geq n_0$ and for each $d$ satisfying $(2r)^{\frac{1}{\epsilon^2}}\leq d \leq n/4$, there exists an $n$-vertex linear $r$-graph with average degree at least $d$ and containing no linear cycles of length at most $(1-\epsilon)\log_d n$.
\end{proposition}

The rest of the paper is organized as follows.
In Section \ref{section:notation}, we introduce some notation.
In Section \ref{section:all-lengths}, we prove Theorem \ref{thm:all-lengths}.
In Section \ref{section:even-lengths}, we prove Theorem \ref{thm:even-lengths},
whose proof is more involved than that of Theorem \ref{thm:all-lengths} due to the tighter control on the shortest lengths of the cycles.
In Section \ref{section:conclusion}, we conclude with some remarks and problems for future study on related topics. 

\section{Notation} \label{section:notation}
Let $r\geq 2$ be an integer. Given an $r$-graph $G$, 
we use $\delta(G)$ and $d(G)$ to denote the minimum degree and the average degree of $G$, respectively.
Given a linear $r$-graph and two vertices $x,y$ in $G$, we define the {\it distance} $d_{G}(x,y)$ to be the length of a shortest linear path between $x$ and $y$. We drop the index $G$ whenever the graph is clear from the context.  For any vertex $x$, we define $L_i(x)$ to be the set of vertices at distance $i$
from $x$. If $x$ is clear in the context we will drop $x$.

Given a graph $G$ and  and a set $S$, an {\it edge-colouring} of $G$ using subsets of $S$ is a function
$\phi: E(G)\to 2^S$. We say that $\chi$ is {\it strongly proper}  if $V(G)
\cap S=\emptyset$ and whenever  $e,f$ are two distinct edges in $G$ that share an endpoint we have $\chi(e)\cap \chi(f)=\emptyset$. 
We say that $\chi$  {\it strongly rainbow} if $V(G)\cap S=\emptyset$ and whenever $e,f$ are distinct edges of $G$  we have $\chi(e)\cap \chi(f)=\emptyset$.

For $r\geq 2$, an $r$-graph $G$ is {\it $r$-partite} if there exists a partition of $V(G)$ into $r$ subsets $A_1,A_2\dots, A_r$ such that
each edge of $G$ contains exactly one vertex from each $A_i$; we call such $(A_1,\dots, A_r)$ an \emph{$r$-partition} of $G$.
For any $1\leq i\neq j\leq r$, we define the $(A_i,A_j)$-{\it projection} of $G$, denoted by $P_{A_i,A_j}(G)$
to be  the graph with edge set $\{e \cap (A_i\cup A_j) |\, e\in E(G)\}$. It is easy to see that for linear $r$-partite $r$-graphs the following mapping  $f: E(G)\to E(P_{A_i,A_j}(G))$ defined by $f(e)=e\cap (A_i\cup A_j)$ is bijective.

Logarithms in this paper are base $2$.


\section{Linear cycles of consecutive lengths} \label{section:all-lengths}
We first  prove some  auxiliary lemmas that are  used in the proof of Theorem~\ref{thm:all-lengths}.
Our first lemma is folklore.

\begin{lemma}\label{lem:folklore} Let $r\geq 2$ be an integer. Every  $r$-graph $G$ of average degree $d$ contains a subgraph of minimum degree  at least $d/r$.
\end{lemma}

\begin{lemma} \label{lem:dense-subgraph1}
Let $r\geq 3$ be an integer. Let $G$ be a linear $r$-graph.  Let $d$ be an integer satisfying $1\leq d\leq  \delta(G)/2$.
Let $x\in V(G)$. 
Then there exist a positive integer $m\leq \lceil \frac{\log n}{\log (\delta(G)/d)}\rceil$
and a subgraph $H$ of $G$ satisfying
\begin{itemize}
\item[(A1)] $H$ has average degree at least $d/4$, and
\item[(A2)] each edge of $H$ contains at least  one vertex in  $L_m(x)$ and no  $\bigcup_{j<m} L_j(x)$.
\end{itemize}
\end{lemma}
\begin{proof}
For each $i>0$, let $G_i$ be the subgraph of $G$ induced by the edges that contain some vertex in $L_i$. Observe that $V(G_i)\subseteq L_{i-1}\cup L_i\cup L_{i+1}$. Let  $t =\lceil \frac{\log n}{\log (\delta(G)/d)} \rceil$.
First we show that for some $i\in [t]$, $G_i$ has average degree at least $d/2$.
Suppose for contradiction that for each $i\in [t]$, $G_i$ has average less than $d/2$. Then for each $i\in [t]$,
$e(G_i)\leq (d/2)|V(G_i)|/r\leq (d/2r) (|L_{i-1}|+|L_i|+|L_{i+1}|)$. On the other hand,  by minimum degree condition we have $e(G_i)\geq \delta(G) |L_i|/r$. Combing the two inequalities, we get
\begin{equation} \label{L-bound}
|L_{i-1}|+|L_i|+|L_{i+1}|\geq \frac{2\delta(G)}{d} |L_i|.
\end{equation}

\medskip

\begin{claim}
For each $i\in [t]$, we have $|L_i|> (\delta(G)/d)|L_{i-1}|$.
\end{claim}
\begin{proof}
The claim holds for $i=1$ since $|L_1|\geq \delta(G)$ and $|L_0|=1$. Let $1\leq j<t$ and
suppose the claim holds for $i=j$. We prove the claim for $i=j+1$. By \eqref{L-bound} and the
induction hypothesis that $L_{j-1}\leq (d/\delta(G))|L_j|$, we have
\[\frac{d}{\delta(G)}|L_j| + |L_j| +|L_{j+1}|\geq \frac{2\delta(G)}{d} |L_j|.\]
Hence
\[|L_{j+1}|\geq (\frac{2\delta(G)}{d}-\frac{d}{\delta(G)}-1)|L_j| > \frac{\delta(G)}{d} |L_j|,\]
where the last inequality uses $d\leq \delta(G)/2$. \end{proof}

By the claim, $|L_t|>(\frac{\log n}{\log (\delta(G)/d)})^t \geq n$, which is a contradiction.
So there exists $i\in [t]$ such that $G_i$ has average degree at least $d/2$. By our earlier discussion,
each edge of $G_i$ contains a vertex in $L_i$ and lies inside $L_{i-1}\cup L_i\cup L_{i+1}$.
If at least half of the edges of $G_i$ contain some vertex in $L_{i-1}$ then let $H$ be the subgraph
of $G_i$ consisting of these edges and let $m=i-1$. Otherwise, let $H$ be the subgraph of $G_i$ consisting
of edges that  do not contain vertices of $L_{i-1}$ and let $m=i$. In either case, $H$ and $m$ satisfy (A1) and (A2).
\end{proof}

\begin{lemma} \label{lem:dense-subgraph2}
Let $	r\geq3$. Let $G$ be a linear $r$-graph.  Let $d$ be a real satisfying $1\leq d\leq  \delta(G)/2$.
Let $x\in V(G)$. For each $v\in V(G)$, let $P_v$ be a fixed shortest $(x,v)$-path in $G$ and
let $\mathcal{P}=\{P_v: v\in V(G)\}$. Then there exist a positive integer $m\leq \lceil\frac{\log n}{\log (\delta(G)/d)}\rceil$, $A\subseteq L_m(x)$ and a subgraph $F$ of $G$
such that the following hold:
\begin{itemize} \item [(P1)]  $\delta(F)\geq d/r2^{2r+1}$,
\item [(P2)] each edge of $F$ contains exactly one vertex from $A$
and no vertices from the set $\bigcup_{j<m} L_j(x)$,
\item[(P3)] for each $v\in V(F)\cap A$, $P_v$ intersects $V(F)$ only in $v$.
\end{itemize}
\end{lemma}
\begin{proof}
By Lemma \ref{lem:dense-subgraph1}, there exist
a subgraph $H$ of $G$  and a positive integer $m\leq \lceil \frac{\log n}{\log (\delta(G)/d)}\rceil$ satisfying properties  (A1)-(A2). So, in particular, $d(H)\geq d(G)/4$.
Now let $X\subseteq V_m$ be obtained by including each vertex of $V_m$ independently with probability $1/2$. We call an edge $f\in E(H)$ \emph{good} if $|f\cap X|=1$. For each such $f\in E(H)$ the probability of it being good is $|f\cap V_m |/2^r\geq 1/2^r$. So there exists a choice of $X$ such that the subgraph of $H$ formed  by the good edges, call it $H'$,   satisfies $e(H')\geq e(H)/2^r$. Fix such a choice of $X$ and the corresponding $H'$. For every edge $f\in E(H')$ let $v_f $ be the unique vertex  in $f\cap X$ and $e_{v_f}$ be the  edge in the path $P_{v_f}$ which contains $v_f$. 

Now, let $Y$ be a random subset of $X$ obtained by choosing each vertex of $X$ independently with probability $1/2$. 
For each edge $f\in E(H')$, we call $f$ {\it nice} if  $e_{v_f}\cap Y=\{v_f\}$. Given any $f\in E(H')$,
the probability of $f$ being nice is $(1/2)^{|e_{v_f}\cap X|}\geq (1/2)^{r-1}$ as $v_f\in e_{v_f}\cap X$ and $|e_{v_f}\cap X |\leq r-1$. So there exists a choice of $Y$ such that the subgraph of $H'$ formed by the nice edges, call it $H''$, satisfies 
\[d(H'')\geq \frac{d(H')}{2^{r-1}} \geq \frac{d(H)}{2^{2r-1}} \geq \frac{d(G)}{2^{2r+1}}.\]
Fix such a choice of $Y$ and $H''$, set $A:= Y$. By Lemma~\ref{lem:folklore} $H''$ has a subgraph $F$ of minimum degree at least $d(H'')/r \geq  d/r2^{2r+1}$. Now, $A$ and $F$ satsify (P1)-(P3).
 \end{proof}

\begin{lemma} \label{lem:path-with-part}
Let $r\geq 3$, $k\geq 1$ be integers. Let $F$ be a linear $r$-graph and $A\subset V(F)$  be such that each edge of $F$ contains
exactly one vertex of $A$. If $\delta(F) \geq r k $ then $F$ contains a linear path of length $k+2$
such that each vertex in $V(P)\cap A$ has degree one in $P$.
\end{lemma}
\begin{proof}
Let $P$ be a longest linear path in $F$ with the property that vertices in $V(P)\cap A$ have degree one in $P$.
Let $e$ be an end edge of $P$. Since $e$ has $r-1\geq 2$ vertices of degree one in $P$ and
$|e\cap A|=1$,  there exists a vertex $v\in e\setminus A$ that has degree one in $P$.
There at least $\delta(G)\geq rk$ edges of $G$ containing $v$. Since $G$ is linear, there are at most
$|V(P)|-r+1$ edges in $G$ that contain $v$ and another vertex on $P$.  Suppose  $|V(P)|-r+1<rk$. Then there is
an edge $f$ in $G$ that contains $v$ and no other vertex on $P$.  But now $P\cup f$ is a longer path than $P$
and each vertex in $V(P\cup f)\cap A$ has degree one in $P\cup f$, contradicting our choice of $P$. Hence
$|V(P)|\geq rk+r-1$, which implies $|P|\geq k+2$.
\end{proof}

\begin{lemma} \label{lem:pan-connected}
Let $r\geq 3, k\geq 1$, $d=kr^2 2^{2r+2}$. Let $F$ be a linear $r$-graph with $\delta(F)\geq 2d$ and $x$ be any vertex in $F$. Then there exist  edges $e$ and $f$ and some integer $t\leq \lceil \frac{\log n}{\log (\delta(F)/d)} \rceil $ such
that for each $i \in \{t+3,t+4\dots, t+k+2\}$ there is a  path of length $i$ starting at $x$ and having $e$ and $f$ as its last two edges.
\end{lemma}
\begin{proof}
For each vertex $v$ in $F$, let $P_v$ be a shortest $(x,v)$-linear path in $F$. By Lemma \ref{lem:dense-subgraph2} (with $F$ playing
the role of $G$) there exist a positive integer $t\leq \lceil \frac{\log n} {\log(\delta(F)/d)} \rceil$, a subset $A\subseteq L_t(x)$ and a subgraph
$F'$ of $F$ that satisfy (P1)-(P3). In particular, $\delta(F')\geq d/r2^{2r+2}=rk$.
Applying Lemma \ref{lem:path-with-part}  to $F'$, we obtain a linear path $P$ of length $k+2$ in $F'$ such that each vertex in $V(P)\cap A$
has degree one in $P$. Suppose the edges of $P$ are ordered as $e_1,\dots, e_k, e, f$. For each $i\in [k]$, let $v_i$ be the unique vertex
in $e_i\cap A$. For each $i\in [k]$ since $P_{v_i}$ intersects $V(F)$ only in $v_i$, $P_{v_i}\cup \{e_i,\dots, e_k,e,f\}$ is
a linear path of length $(k+2)-(i-1)+t$ that starts at $x$ and ends with $e,f$. Since this holds for each $i=1,\ldots, k$, the claim follows.
\end{proof}

Now we are ready to prove Theorem~\ref{thm:all-lengths}.

\medskip

{\bf Proof of Theorem~\ref{thm:all-lengths}:} We will show the statement holds for $c_1=2^{4r+8} r^3$, $c_2=\log{c_3}$, where $c_3=2^{4r+4}r^5$. By Lemma~\ref{lem:folklore} $G$ contains a subgraph of minimum degree $d(G)/r$. With some abuse of notation, let us denote that subgraph by $G$ as well and let $\delta=d(G)/r$. Set $d'=kr^2 2^{2r+2}$, $d= r^{3/2} 2^{2r+2}\sqrt{\delta k}$.

Let $x_0$ be any vertex in $G$. By Lemma~\ref{lem:dense-subgraph2}  there exist  $m\leq \lceil \frac{\log n}{\log (\delta/d)} \rceil$,
a subset $A\subseteq L_m(x_0)$ and a subgraph $F$ of $G$ such that
\begin{enumerate}
\item [(P1)] $\delta(F)\geq  \frac{d}{r2^{2r+2}}$,
\item [(P2)] each edge of $F$ contains exactly one vertex in $A$ but no vertex in $\bigcup_{j<i} L_j(x_0)$, and
\item [(P3)] for each $v\in V(F)\cap A$, $P_v$ intersects $V(F)$ only in $v$.
\end{enumerate}

Now  let $x$ be any vertex in $V(F)\cap A$. 
Since $\delta(F)\geq d/r2^{2r+2}\geq 2d'$,
by Lemma \ref{lem:pan-connected},  there exist two edges $e$ and $f$ in $F$ and some integer
$t \leq \lceil \frac{\log n}{\log (\delta(F)/d')}\rceil$ such that for each $i\in \{t+3, t+4, \dots, t+k+2\}$
there is a linear a path $Q_i$ in $F$ of length $i$ which starts at $x$ and has $e$ and $f$ as the last two edges.

Let $y$ be the unique vertex in $A\cap f$.  By (P3),  $P_{x}$ and $P_{y}$ intersect $ V(F)$ only in $x$ and $y$, respectively. Therefore, $P_{x}\cup P_y$ must contain a linear $(x,y)$-path of length $q\leq 2m$ that intersects $V(F)$ only in $x$ and $y$. Let us denote this subpath by $P_{xy}$.

If $y\notin e\cap f$, then $P_{xy}\cup Q_1,\ldots, P_{xy}\cup Q_k$ are linear cycles of lengths
$q+t+3, \ldots, q+t+k+2$, respectively. If $y\in e\cap f$ then $P_{xy}\cup (Q_1\setminus f), \ldots, P_{xy}\cup \{Q_k\setminus f)$
are linear cycles of lengths $q+t+2,\ldots, q+t+k+1$, respectively.
In either case we find linear cycles of $k$ consecutive lengths, the shortest of which has length at most

\begin{eqnarray*}
2m+t+3 &\leq& 2 \left (\frac{\log n}{\log (\delta/d)} +1\right)  + \left ( \frac{\log n}{\log (\delta(F)/d')}+1\right )  +3  \\
&\leq & 3 \frac{\log n}{\log (\delta/d)} +6, \\
\end{eqnarray*}

where the last inequality holds since $\delta(F)/d'\geq \delta/d$. To conclude the proof, just note that $(d(G)/kc_3)^{1/2}\leq \delta/d$, by our choice of $d$ and $c_3$. Therefore,
 the shortest length of a cycle in the collection is at most $6\lceil\frac{\log n}{\log (d(G)/c_3k)}\rceil +6$.
\qed


\section{Sharper results for linear cycles of even consecutive lengths}
\label{section:even-lengths}

For linear cycles of even consecutive lengths, we obtain much tighter control on the shortest length of
a cycle in the collection, which as a byproduct also gives us an improvement on the current best known upper bound on
the linear Tur\'an number $\ex_L(n,C^{r}_{2k})$ of an $r$-uniform linear cycle of a given  even length $2k$.
The previous best known upper bound is $c_{r,k} n^{1+1/k}$, where $c_{r,k}$  is
exponential in $k$ for fixed $r$. For fixed $r$, we are now able to improve the bound on $c_{r,k}$ to a linear function of $k$.

\subsection{A  useful lemma on long paths with special features}

One of the key ingredients of our proof of the main result in this section is Lemma \ref{lem:path-JMY}.
The lemma is about the existence
of a long path with special features in an edge-colored graph with high average degree.
It may be viewed a strengthening of two lemmas used in \cite{JM} (Lemma 2.6 and Lemma 2.7).
We start with a preliminary lemma.

\begin{lemma} \label{degenerate-ordering}
Let $G$ be connected graph with average degree at least $2d$.  Then there exists a linear ordering $\sigma$ of
$V(G)$ as $x_1< x_2< \dots< x_n$ and some $0\leq m<n$ such that for each $1\leq i\leq m$ $|N_G(x_i)\cap \{x_{i+1},\dots, x_n\}|< d$
and that the subgraph $F$ of $G$ induced by $\{x_{m+1},\dots, x_n\}$ has minimum degree at least $d$.
\end{lemma}
\begin{proof}
As long as $G$ contains a vertex whose degree in the remaining subgraph is less than $d$ we delete it from $G$. We continue until no such vertex exists. Let $F$ denote the remaining subgraph. Suppose this terminates after $m$ steps. Then we have deleted at most 
$dm\leq d(n-1)<e(G)$ edges. Hence $F$ is nonempty. Let $x_1< x_2<\dots< x_m$ be the vertices deleted in that order. Let $x_{m+1}<
\ldots< x_n$ be an arbitrary linear ordering of the remaining vertices. Then the ordering $\sigma:=x_1<\ldots< x_n$ and $F$ satisfy the requirements.
\end{proof}

The following lemma is written in terms of  colourings of graphs, but in our applications $H$ will be some $(A_i,A_j) $-projection of an $r$-partite $r$-graph $G$ where the colouring is obtained by colouring  the edge $e\cap (A_i\cup A_j)$ (where $e\in E(G)$) in $H$ by the $(r-2)$-set $e\setminus (A_i\cup A_j)$.

\begin{lemma} \label{lem:path-JMY}
Let $r\geq 3$. Let $H$ be a connected graph with minimum degree at least $4r\ell$. Let $\chi$ be a strongly proper edge-colouring of $H$ using $(r-2)$-sets. Let $E_1, E_2$ be any partition of $E(H)$ into two nonempty sets such that $|E_1|\leq |E_2|$.
Then there exists a strongly rainbow path of length at least $\ell$ in $H$ such that the first edge of $P$ is in $E_1$ and
all the other edges are in $E_2$.
\end{lemma}
\begin{proof}
For $i=1,2$, let $H_i$ be the subgraph of $H$ induced  by the edge set $E_i$. Note that $d(H_2)\geq 2r\ell$. Let $L$ be a connected component of $H_2$ with $d(L)\geq 2r\ell$. By Lemma~\ref{degenerate-ordering}, there exist
$0\leq m<n$ and $\sigma:=x_1<x_2<\dots< x_n$  be as in Lemma~\ref{degenerate-ordering} such that
for each $1\leq i\leq m$, $N_L(x_i)\cap \{x_{i+1},\dots x_n\}|<r\ell$ and that the subgraph $F$ of $H$
induced by $\{x_{m+1},\dots, x_n\}$ has minimum degree at least $r\ell$.

Let us call a strongly rainbow path $P$ in $H$ a {\it good path} if it has length at least one, its first edge is in $E_1$
and its other edges (if exist) are in $E_2$. To prove the lemma, we need to show that $H$ has a good path of length $\ell$.

\begin{claim} \label{claim:good-path}
If $H$ has a good path that ends with a vertex in $F$ then $H$ has a good path of length $\ell$.
\end{claim}
\begin{proof}
Among all good paths in $H$ that end with a vertex in $F$, let $P$ be a longest one.
If $|P|\geq \ell$ then we are done. Hence we may assume that $P=uv_1v_2\dots v_j$. for some $j\leq \ell-1$. Since $\delta(L)\geq r\ell$,
there are at least $r\ell$ edges of $L$ incident to $v_j$. Among these edges, more than $\ell  r - j > \ell (r-1) $ of them join $v_j$ to a vertex outside $V(P)$. Since the colouring $\chi$ is strongly proper, the colours of these edges form a  matching  of $(r-2)$-sets of size more than $\ell(r-1)$.
Let $C(P)=\bigcup_{e\in E(P)}\{c|c\in \chi(e)\}$. Then $|C(P)|\leq j(r-2)< \ell(r-2)$.
Hence, there must exist  a vertex $v_{j+1}\in V(L)$
outside $V(P)$ such that $\chi(v_jv_{j+1})\cap C(P)=\emptyset$.
Now, $P\cup v_jv_{j+1}$ is a longer good path than $P$, a contradiction.
\end{proof}

Let us call a path $x_{j_1}x_{j_2}\dots x_{j_t}$ in $L$ an {\it increasing path} under $\sigma$ if  $x_{j_1}<x_{j_2}<\dots< x_{j_t}$ in $\sigma$;
we call $x_{j_t}$ the {\it last vertex} of the path. Let $\mathcal{P}$ be the collection of  strongly rainbow increasing paths in $L$ with the property that either it has length $\ell-1$ or  it has length less than $\ell-1$ and its last vertex is in $F$. 
 As a single vertex in $F$ is an increasing path, $\mathcal{P}\neq\emptyset$.
Among all the paths in $\mathcal{P}$  let $P= x_{j_1}x_{j_2}\dots x_{j_t}$ be such that $j_1$ is minimum.  
By our assumption, either $t=\ell$ or $t<\ell$ and $x_{j_t}\in V(F)$.
If  $|N_L(x_{j_1})\cap \{x_1, \dots, x_{j_1-1}\}|>\ell (r-2)$ then by a similar argument as in the proof of Claim \ref{claim:good-path} we can find $j_0<j_1$ such that $\chi(x_{j_0}x_{j_1})$ is disjoint from all $\chi(x_{j_i}x_{j_{i+1}})$ for all $i\in [t-1]$  and $x_{j_0}x_{j_1}\in L$.
In this case,  either $x_{j_0} x_{j_1}\dots x_{j_{t-1}}$ or $x_{j_0}x_{j_1}\ldots x_{j_t}$ would contradict our choice of $P$. Hence,  $|N_L(x_{j_1})\cap \{x_1, \dots, x_{j_1-1}\}|\leq \ell (r-2)$. By the definition of $\sigma$,
$|N_L(x_{j_1})\cap\{x_{j_1+1},\ldots, x_n\}|<r\ell$. Hence, $d_L(x_{j_1})<\ell(r-2)+r\ell<3r\ell$. Since $\delta_H(x_{j_1})\geq 4r\ell$, $x_{j_1}$ is incident to at least  $4\ell r- \ell r- \ell (r-2)=2\ell (r+1)  $ many  edges in $E_1$. Among them more than $2\ell r +\ell$ of them joins $x_{j_1}$
to a vertex outside $V(P)$.
Since $\chi$ is strongly proper, the colours on these edges form
a matching of size more than $2\ell(r+1)$. Since $C: =\bigcup_{i=1}^t \chi(x_{j_i}x_{j_{i+1}})$ has size less than
$\ell r$,
there must exist at least one edge of $E_1$ that joins $x_{j_1}$ to a vertex $x_{j_0}$ outside $V(P)$ 
such that $\chi(x_{j_0}x_{j_1})$ is disjoint from $C$. Now, $x_{j_0}x_{j_1}\ldots x_{j_t}$ is a good path of length $t+1$.
If $t=\ell$ then we are done. If $t<\ell$, then $x_{j_t}\in V(F)$ and we are done by Claim \ref{claim:good-path}.
\end{proof}

The following cleaning lemma is similar to part of Lemma \ref{lem:dense-subgraph2}.

\begin{lemma} \label{lem:transversal-subgraph}
Let $H$ be a linear $r$-partite $r$-graph with an $r$-partition $(A_1,\cdots, A_r)$.
Let  $M$ be  an $(r-1)$-uniform matching where for each $f\in M$, $f$ contains one vertex of
each of $A_2,\dots, A_r$. Then there exists a subgraph $H'\subseteq H$ such that
\begin{itemize}
\item [(1)] $e(H')\geq [1/(r-1)]^{r-1} e(H)$,
\item [(2)] each edge of $M$ intersects $V(H')$ in at most one vertex.
\end{itemize}
\end{lemma}
\begin{proof}
Let us independently colour each edge of $M$ using a colour in $\{2,\dots, r\}$ chosen uniformly at random.
Denote the colouring $c$.
For each $i\in \{2,\dots, r\}$, let $M_i=\{f\in M: c(f)=i\}$ and  let $B_i=\{f \cap A_i: f\in M_i\}$.
Let $H'=\{e\in E(H): e\cap V(M) \subseteq B_2\cup\dots \cup B_r\}$.

Let $f$ be any edge of $M$.  By the definition of $H'$,
$f\cap V(H')\subseteq B_2\cup\dots \cup B_r$. Suppose $f$ is coloured $i$. Then since $M$ is a matching, we have
 $|f\cap B_i|=1$ and $f\cap B_j=\emptyset$ for each $j\in \{2,\ldots, r\}\setminus \{i\}$.
 Therefore, $|f\cap V(H')|=1$.

Next, for some colouring $c$ the resulting $H'$ satisfies $e(H')\geq [(1/(r-1)]^{r-1} e(H)$.
Let $e$ be any edge of $H$. Let $s=|e\cap V(M)|$.  If $s=0$ then
 $e$ is in $H'$ with probability $1$. So we may assume that $1\leq s \leq r-1$.
 Since $G$ is $r$-partite, the $s$ vertices of $S$ all lie in different parts among $A_2,\ldots, A_r$.
 Without loss of generality, suppose $e\cap V(M)=\{a_2,\dots, a_{s+1}\}$, where for each $i=2,\dots, s+1, a_i\in A_i$.
 Since $M$ is matching, for each $i=2,\ldots, s+1$, there is a unique edge $f_i\in M$ that contains $a_i$.
 The probability that $a_i\in B_i$ is the probability that $f_i$ is coloured $i$, which is  $1/(r-1)$. 
 Hence, the probability that for each $i=2,\ldots, s+1$,
 $a_i\in B_i$ is $[1/(r-1)]^{r-1}$. In other words, the probability that $e$ is in  $H'$ is 
 $[1/(r-1)]^s\geq [1/(r-1)]^{r-1}$. So there exists a colouring $c$ for which
$e(H')\geq [1/(r-1)]^{r-1} e(H)$.  The subgraph $H'$ satisfies the requirements of the lemma.
\end{proof}

\subsection{Rooted expanded trees and linear cycles of consecutive even lengths}

In this subsection, we introduce some of the key notions we use, in particular, a variant of a breadth-first-search tree in a linear $r$-partite $r$-graph $G$, and prove some auxiliary results we need for the proof of the main theorem.
 
\begin{definition} \label{rooted-tree}
Let $r\geq 3$ be an integer. Let $G$ be a graph.  Let $\phi$ be any edge-colouring $\chi$ by $(r-2)$-sets satisfying that for every edge $e=uv$ we have $u,v\notin \chi(uv)$. We define the \emph{$(\chi,r)$-expansion} of $G$, denoted by
$G^{\chi}$, to be the $r$-graph on vertex set $V(G)\cup \chi(G)$ obtained from $G$
by expanding each edge $e$ of $G$ into the $r$-set $e\cup \chi(e)$, where $\chi(G)=\{c\in \chi(e) \textit{  for some } e\in E(G)\}$. 

In the definition of $(\chi,r)$-expansion we don't require the sets $V(G)$ and $\chi(G)$ to be disjoint.  However, if $\chi$ is a strongly rainbow, then $(\chi,r)$-expansion is isomorphic to what is known in the literature, as the \emph{$r$-expansion} of $G$, defined as follows. The {\it $r$-expansion} $G^r$ of $G$ is an $r$-graph obtained from $G$
by expanding each edge $e$ of $G$ into an $r$-set using pairwise distinct $(r-2)$-sets  disjoint from $V(G)$. 
Note that the $(r-2)$-sets  used for the expansion naturally define a strongly rainbow edge-colouring on $G$.
\end{definition}

\begin{algorithm} {\bf (Maximal Expanded  Rooted Tree - MERT)} \label{maximal-tree}
{\rm

\medskip

{\bf Input:} A linear $r$-partite $r$-graph $G$ with a fixed $r$-partition $(A_1,\dots, A_r)$ and a vertex $x$ in $A_1$.

{\bf Output: } $(H,T,\chi)$ where $H$ is some subgraph $H\subseteq G$, $T$ is a tree rooted at $x$ such that $H$ is the $r$-expansion of $T$ and furthermore, for each $i\geq 0$,
there exists some $j\in [r]$ such that $L_i(x)\subseteq A_j$, where $L_i(x)$ is the $i$th level in $T$, and finally $\chi$ is a strongly rainbow edge-colouring of $T$. 

We will also obtain a collection of subgraphs of $H$, $\{H_i\}_{i=0}^m$ where each $H_i$ is called the $i$th segment of $H$ and a collection of $(r-1)$-uniform matchings $\{M_i\}_{i=1}^m$ where $V(M_i)\subset V(H_i)\setminus V(H_{i-1})$ and $M_i$ is called the $i$th matching of $H$, these are described further below.

{\bf Initialization:} Let $H_0=\{x\}$. Let $L_0=\{x\}$, $T_0=\{x\}$. Let $H_1$ be the subgraph of $G$ consisting of all the edges
of $G$ containing $x$. For every $v\in I\in E(H_1)\setminus\{x\}$ let $p_v=\{x\}$.

{\bf Iteration:} Let $E_i$ denote the set of edges in $G$ that contain exactly one vertex in $V(H_i)\setminus V(H_{i-1})$.  If $E_i=\emptyset$ then let $L_i=(V(H_i)\setminus V(H_{i-1}))\cap A_2$,
and let $T_i$ be the super-tree of $T_{i-1}$ obtained from $T_{i-1}$ by joining every $v\in L_{i}$ to $p_v\in L_{i-1}$. Let $H= \cup_{0\leq j\leq i} H_i$, $T=T_i$ and terminate.

If $E_i\neq \emptyset$ then do the following. Suppose $L_{i-1}\subseteq A_\ell$. For each $j\in [r]\setminus \{\ell\}$, let $E_i^j$
be the set of edges in $e\in E_i$ such that
$|e\cap (V(H_i)\setminus V(H_{i-1}))\cap A_j|=1$. Then $E_i=\bigcup_{j=\in [r]\setminus \{\ell\}} E_i^j$. Let $s(i)$ be some $j\in [r]\setminus \{\ell\}$ that maximizes $|E_i^j|$.
Let $L_i= E_i^{s(i)} \cap A_{s(i)}$.
Let $M_i$ be a largest matching of $(r-1)$-tuples
in $\{e\setminus L_i: e\in  E_i^{s(i)}\}$.
For each $I\in M_i$ we do the following.
Since the graph $G$ is linear, there  is a unique $v_I\in L_i$ such that $I\cup v_I\in E_i^{s(i)}$. For each $u\in I$, we define $p_u$ to be $v_I$ and refers to it as the {\it parent} of $u$. 
Let $H_{i+1}$ be the subgraph of $G$
induced by the edges $\{I\cup v_I|I\in M_i\}$.   Increase $i$ by one and repeat. 

\textbf{Stop:} Suppose the algorithm stopped after $m$ steps then we call $m$ the \emph{height} of $H$, noting that $m$ is also the height of the tree $T$. We will interchangeably  call both $H$ and the pair $(H,T)$ an \emph{MERT} of $G$ rooted at $x$. Let $\chi$ be the following colouring on $T$: For every edge $uv\in E(T)$ there is a unique $(r-2)$-tuple $I$ such that $uv\cup I\in E(H)$, we let $\chi(uv)=I$. By construction of $H$, $\chi$ is strongly rainbow.
\qed

}
\end{algorithm}

\begin{figure}[H]
\begin{center}
  \includegraphics[height=7.5cm]{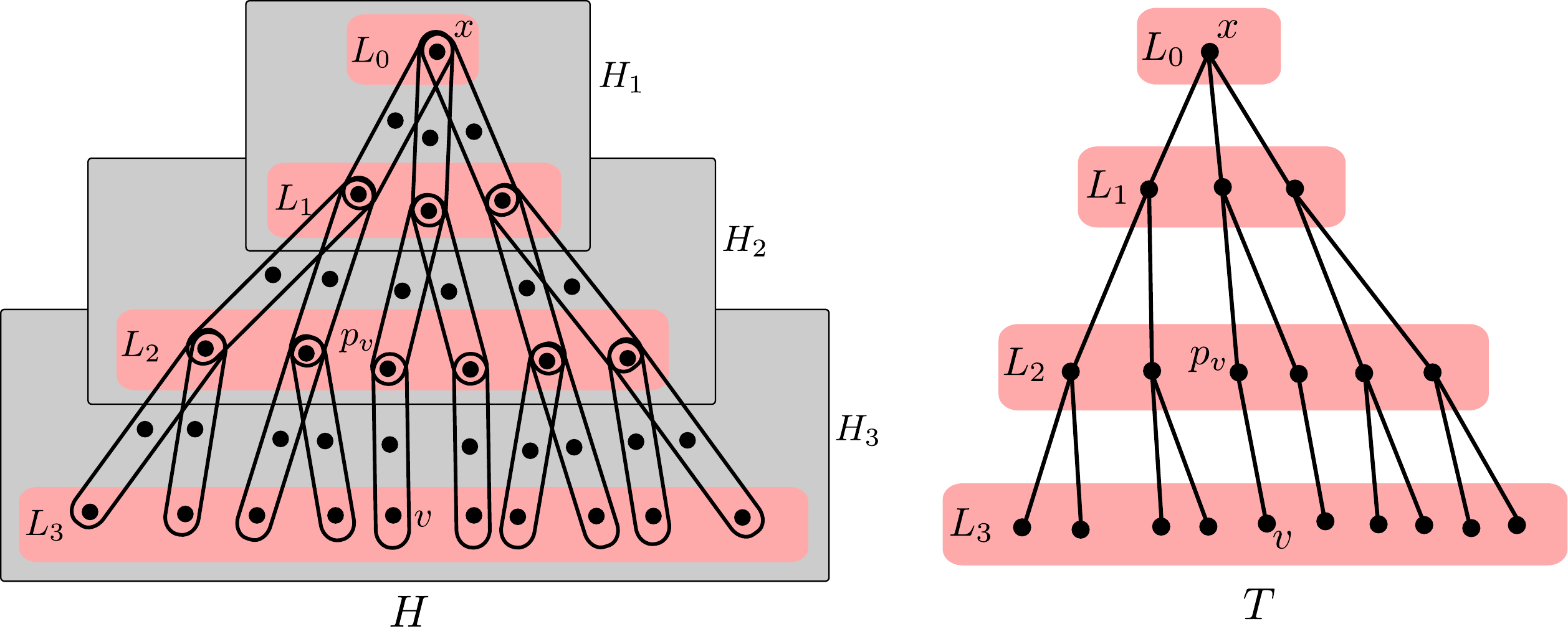}
 \label{fig:mert}
  \end{center}
  \caption{$H=H_1\cup H_2\cup H_3$ and the corresponding tree $T$}
\end{figure}

\begin{lemma} \label{lem:tree+dense1}
Let $r\geq 3,t\geq 1$. Let $G$ be an $r$-partite $r$-graph  with an $r$-partition $(A_1, A_2,  \dots, A_r)$. Let $x$ be a vertex
in $G$. Let $(H,T)$ be an MERT rooted at $x$.  Let $D$ be the subgraph of $G$ consisting of all the edges in $G$ that contain a vertex in $L_{t-1}$, at least
  one vertex in $V(H_t)\setminus L_{t-1}$ and no vertices from $(\bigcup_{j<t} (V(H_j)\setminus L_{t-1}))$.
If $e(D)\geq 8kr(r-1)(|L_{t-1}|+|L_t|)$ then $G$ contains linear cycles of lengths $2m+2, 2m+4,\dots, 2m+2k$ for some $m\leq t-1$.
\end{lemma}

\begin{proof} By definition of MERT, without loss of generality we may suppose $L_{t-1}\subseteq A_1$.
By definition,  each edge of $D$ contains a vertex in $L_{t-1}$ and at least one vertex in
$V(H_t)\setminus L_{t-1}$.  
Since $A_2\cap V(H_t), A_3\cap V(H_t),\dots, A_r\cap V(H_t)$ partition $V(H_t)\setminus L_{t-1}$,
by the pigeonhole principle, for some $i\in \{2,\dots, r\}$,
at least $e(D)/(r-1)$ of the edges of $D$ contain a vertex from  $A_i\cap V(H_t)$.  Without loss of generality, suppose $i=2$.

Let $X=L_{t-1}$ and $Y=A_2\cap V(H_t)$. By definition of MERT, $|V(H_t)\cap A_2|=|L_t|$, so $|Y|=|L_t|$.
Let $D'$ be the subgraph of $D$ consisting of the edges that contain a vertex in $X$ and a vertex in $Y$. By the previous discussion,\begin{equation} \label{D'-lower}
e(D')\geq e(D)/(r-1).
\end{equation}
Let $B$ be the $(X, Y)$-projection of $D'$. Since $G$ is linear, $e(B)=e(D')$.
Also, $|V(B)|\leq |X|+|Y|= |L_{t-1}| + |L_t|$.  By our assumption about $e(G)$ and  \eqref{D'-lower},
\[e(B)\geq 8kr(|L_{t-1}+|L_t|)\geq 8kr |V(B)|.\]
So $B$ has average degree at least $16kr$. By a well-known fact, $B$ contains a connected subgraph $B'$
with minimum degree at least $8kr$.

Let $S=V(B')\cap X$. Suppose $x'$ is the closest common ancestor of $S$ in $T$. The union of the paths of $T$ joining vertices of $S$ to $x'$
forms a subtree $T_S$ of $T$ rooted at $x'$. Suppose that $x'\in L_j$. Then $V(T')\subseteq L_j\cup \dots \cup L_{t-1}$, $x'$ is the only vertex in $V(T')\cap L_j$.  For each $v\in S$, let $P_{v,x'}$ denote the unique $(v,x')$-path in $T'$.


Since $x'$ is the closest common ancestor of $S$ in $T$,
$x'$ has at least two children in $T'$. Let $x_1$ be one of the children of $x$ in $T'$.  We define a vertex labelling $f$ on $S$ as follows.
For each $v\in S$,  if $P_{x',v}$ contains $x_1$ then let $f(v)=1$, and otherwise let $f(v)=2$. Note that since $x$ had at least two children, there will be some $u, v\in S$ with $f(u)=1$ and $f(v)=2$. The following claim is one of the key ingredients used by Bondy and Simonovits in proving their results in~\cite{BS74}. For completeness,
we include a proof. 

\begin{claim} \label{Quv}
Let $u,v\in S$. If $f(u)=1$ and $f(v)=2$ then $P_{x',u}\cup P_{x',v}$ is a path of length $2(t-1-j)$ in $T'$
that intersects $S$ only in $u$ and $v$. \qed
\end{claim}

\begin{proof} It is clear that $V(P_{u,x'})\cap S=\{u\}$ and $V(P_{v,x'})\cap S=\{v\}$. To see that $P_{u,x'}$ and $P_{v,x'}$ only intersect at $x'$, suppose otherwise. Recall that since $f(u)=1$, the path $P_{u,x'}$ contains $x_1$ and $P_{v,x'}$ does not. Let $y$ be the first vertex on $P_{u,x'}\cap P_{v,x'}$ along the path $P_{u,x'}$. By our assumption $y\neq x'$ ($y$ could be $x_1$). Let $P_1$ be the subpath of $P_{v,x'}$ that goes from $v$ to $y$, and let $P_2$ be the subpath of $P_{u,x'}$ from $y$ to $x'$. It is easy to see that that $P=P_1\cup P_2$ is an $(v,x')$-path in $T'$ and furthermore, $P$ does not go through $x_1$
and hence must equal to $P_{v,x'}$. But $P_2$ and hence $P$ goes through $x_1$ since $f(u)=1$, which contradicts to $f(v)=1$.
\end{proof}


Now,  we define a partition of $E(B')$ into $E_1$ and $E_2$ as follows.  
Let $ab$ be any edge in $E(B')$ where $a\in X$ and $b\in Y$. For $i=1,2$, we put $ab$ in $E_i$ if $f(a)=i$. We  define an edge-coloring $\varphi$ on $B'$ using $(r-2)$-sets by letting $\varphi(ab)$ be the unique $(r-2)$-set such that $ab\cup \varphi(ab)\in E(D')$ for all $ab\in E(B')$. Since $G$ is linear, $\varphi$ is strongly proper. By Lemma \ref{lem:path-JMY}, with $\ell=2k$, $B'$ contains a strongly rainbow path $P$ of length $2k$ such that the
first edge of $P$ is in $E_1$ and all the other edges of $P$ are in $E_2$.
Suppose $P=a_1b_1a_2b_2\dots a_kb_ka_{k+1}$. Note that we must have $a_1\in S$. Otherwise
if $b_1\in S$ instead then the first two edges of $P$ would have the same colour, contradicting
our definition of $P$. Hence, $a_1,a_2,\dots, a_{k+1}\in S$ and $b_1,b_2,\dots, b_k\in A_2$.
Also, by our assumption about $P$, $f(a_1)=1$ and $f(a_2)=\cdots=f(a_{k+1})=2$. For each $i\in [k]$, let $P_i$ be the subpath $P$ from $a_1$ to $a_i$ and let $Q_{a_i}$ denote the unique path in $T'$ from $x'$ to $a_i$.
Let $\chi$ be the colouring in $(H,T,\chi)$ produced by Algorithm \ref{maximal-tree}.

\begin{claim} \label{even-pasting}
 For each $i\geq 2$, let $R_{i}$ be the union of the $r$-uniform paths $P_{i}^{\varphi}$, $Q_{a_1}^{\chi}$ and $Q_{a_{i}}^{\chi}$. Then $R_{i}$ is a linear path of length $2(t-1-j)+2(i-1)$ in $G$.
\end{claim}
\begin{proof}
Since $f(a_1)=1$ and $f(a_{i})=2$,  by Claim \ref{Quv}, $Q_{a_1} \cup Q_{a_{i}}$ is a path
of length $2(t-1-j)$ in
$T'$ that intersects $S$ only in $a_1$ and $a_{i}$. On the other hand $P_i$
is path of length $2(i-1)$ in $B'$. So it intersects $Q_{a_1}\cup Q_{a_{i}}$ only at $a_1$ and $a_{i}$.
So $P_i\cup Q_{a_1}\cup Q_{a_{i}}$ is a cycle of length $2(t-1-j)+2(i-1)$  in $T'\cup B'$.
By our assumptions, $\varphi$ is strongly rainbow on $P_i$  and $\chi$ is strongly rainbow on
$Q_{a_1}\cup Q_{a_i}$.
Furthermore, for any $e\in P_i$ and $f\in Q_{a_1}\cup Q_{a_{i}}$, $\varphi(e)\in V(H_t)$ while $\chi(f)\in \bigcup_{j<i} V(H_j)\setminus L_{t-1}$. So $\varphi(e)\cap \chi(f)=\emptyset$. Therefore, $R$ is a linear cycle of length $2(t-1-j)+2(i-1)$ in $G$.
\end{proof}

By Claim \ref{even-pasting}, the lemma holds for $m=t-1-j$.
\end{proof}

In the next lemma, we in fact obtain linear cycles of consecutive lengths, instead of just consecutive even lengths.

\begin{lemma} \label{lem:tree+dense2}
Let $r\geq 3,t\geq 1$. Let $G$ be an $r$-partite $r$-graph  with an $r$-partition $(A_1, A_2,  \dots, A_r)$. Let $x$ be a vertex
in $G$. Let $(H,T)$ be an MERT rooted at $x$. 
Let $t\geq 1$. 
  Let
\[F=\{e\in E(H): e\cap \bigcup_{i<t} V(H_i) =\emptyset \mbox{ and }  |e\cap V(H_t)|\geq 2|\}.\]
If $e(F)\geq 8k r^{r+2} |L_t|$ then $G$ contains linear cycles of lengths $2m+1, 2m+2,\dots, 2m+2k$, respectively,
for some $m\leq t$.
\end{lemma}
\begin{proof} By our assumption $L_{t-1}$ is contained in one partite set of $G$. Without loss
of generality suppose that $L_{t-1}\subseteq A_1$. Then $V(H_t)\setminus L_{t-1}\subseteq A_2\cup \dots \cup A_r$.
Let $M_t=\{e\setminus L_{t-1}: e\in H_t\}$. Since $H$ is an $r$-expansion of $T$, it is easy to see that
$M_t$ is an $(r-1)$-uniform matching contained in $A_2\cup \dots \cup A_r$.
By Lemma \ref{lem:transversal-subgraph}, there exists a subgraph $F'$ of $F$ such that
\begin{enumerate}
\item  $e(F')\geq (1/(r-1))^{r-1} e(F)$,
\item \label{at-most-one} each edge of $M_t$ intersects $V(F')$ in at most one vertex.  \label{transversal-property}
\end{enumerate}
Since $V(F')$ is disjoint from $L_{t-1}$, item 2 above ensures that
\begin{equation}\label{eq:at-most-one}
\forall e\in H_t, \,  |e\cap V(F')|\leq 1.
\end{equation} 

Let $e$ be any edge of $F'$.
By the definition of $F$ and the fact that $F'\subseteq F$, $e$ contains at least two vertices of $V(H_t)=V(M_t)$. Also, since $(A_1,\dots, A_r)$ is an $r$-partition of $G$ and $V(M_t)\subseteq A_2\cup\dots \cup A_r$,  there
exists a pair $(i,j)$ in $\{2,\dots, r\}$ such that $|e\cap V(M_t)\cap A_i|=|e\cap V(M_t)\cap A_j|=1$.
By the pigeonhole principle, for some $i,j\in \{2,\ldots,r\}$ 
the subgraph $F''$ of $F'$ with edge set
$\{e\in E(F'): |e\cap V(M_t)\cap A_i|=|e\cap V(M_t)\cap A_j|=1\}$
satisfies 
\[e(F'')\geq e(F')/\binom{r-1}{2} \geq (2/r^{r+1}) e(F).\]

By our condition on $F$, $e(F)\geq 8kr^{r+2}|L_t|$.
Hence
\begin{equation} \label{F''-lower}
e(F'')\geq 16kr|L_t|.
\end{equation}

Without loss of generality, suppose that $i=2, j=3$.
Let $B$ be the $(A_2, A_3)$-projection of $F''$.
Since $G$ is linear, $e(B)=e(F'')$. Also, note that $|V(B)|
\leq |V(M_t)\cap A_2|+|V(M_t)\cap A_3|\leq 2|L_t|$. Hence, by \eqref{F''-lower},
\[e(B)=e(F'')\geq 16kr|L_t|\geq 8kr |V(B)|.\]
So $B$ has average degree at least $16kr$. By a well-known fact, $B$ contains a connected subgraph $B^*$ such that
\[\delta(B^*)\geq 8kr.\] Let $F^*$ be the subgraph of $F''$ such that
the $(A_2,A_3)$-projection of $F^*$ is  $B^*$.  Let $S=V(F^*)\cap L_{t-1}$.
Let $x'$ be the closest common ancestors of $S$ in $T$. Let $T_S$ be the
subtree formed by the paths in $T$ from $S$ to $x'$.
Suppose that $x'\in L_j$. Then $V(T_S)\subseteq L_j\cup \dots \cup L_{t-1}$ and that $x'$ is the only vertex in $V(T_S)\cap L_j$. Furthermore, the minimality of $T_S$ implies
that $x'$ has at least two children in $T_S$. For each $v\in S$, let $P_{x',v}$ denote the unique $x',v-$path in $T_S$.

Now we define a labelling $f$ of vertices in $S$ as follows.
Let $x_1$ be one child of $x$ in $T'$.
For each $v\in S$,  if $P_{x',v}$ contains $x_1$ then let $f(v)=1$;  otherwise let $f(v)=2$.
As in the proof of Lemma \ref{lem:tree+dense1},
the definitions of $T_S$ and $f$ ensure the following.

\begin{claim} \label{Quv2}
Let $u,v\in S$. If $f(u)=1$ and $f(v)=2$, then $P_{x',u}\cup P_{x',v}$ is a path of length $2(t-1-j)$ in $T_S$
that intersects $S$ only in $u$ and $v$. \qed
\end{claim}

For each vertex $y\in V(B^*)$, there is a unique edge $e_y$ of $H_t$ that contains $y$.
Let $v_y$ be the unique vertex in $e_y\cap L_{t-1}$.
We now partition $E(B^*)$ into $M$ and $N$ as follows. Let
\[M=\{ab\in E(B^*): f(v_a) =  f(v_b)\} \quad \mbox { and } \quad N=\{ab\in E(B^*): f(v_a)\neq f(v_b)\}.\]
Let us define an edge-colouring $\varphi$ of $B^*$ using $(r-2)$-sets as follows.
For each $ab\in E(B^*)$, let $\varphi(ab)$ be the unique $(r-2)$-set such that $ab\cup \varphi(ab) \in E(F'')\subseteq E(G)$ for all $ab\in E(B^*)$. Since $G$ is $r$-partite, $\varphi(B^*)$ is disjoint from $V(B^*)$. Since $G$ is linear, $\varphi$ is strongly proper. There are two cases to consider.

\medskip

{\bf Case 1.} $|M|\geq |N|$.

\medskip

Applying Lemma \ref{lem:path-JMY} with $E_1=N, E_2=M$, $\ell=2k$, there exists a strongly rainbow path (under $\varphi$) $P=ab_1b_2\dots b_{2k}$ of of length $2k$
in $B^*$ such that the first edge is in $N$ and all the other edges are in $M$. Let us assume that $f(v_a)=1$; the case $f(v_a)=2$ can be argued similarly. Since $ab_1\in N$, we have $f(v_{b_1})=2$. Since $b_ib_{i+1}\in M$ for $i=1,\dots, 2k-1$, we have $f(v_{b_1})=\cdots= f(v_{b_{2k}})= 2$.
Let $Q_{v_a}$ denote the unique path in $T_S$ from $x'$ to $v_a$. For each 
$i\in [2k]$ let $P_i$ denote the portion of $P$ between $a$ and $b_i$ and let $Q_{v_{b_i}}$ denote 
the unique path in $T_S$ from $x'$ to $v_{b_i}$.
Since $f(v_a)=1, f(v_{b_i})=2$, by Claim \ref{Quv2} $Q_{v_a}\cup Q_{v_{b_i}}$ is a path of length $2(t-1-j)$ in $T_S$.

\begin{claim}
For each $i\geq 1$ let $R_i$ be the union of
the $r$-uniform paths $P_i^\varphi$, $Q_{v_a}^\chi$, $Q_{v_{b_i}}^\chi$ and $\{e_a, e_{b_i}\}$. Then $R_i$ is a linear cycle of length $2(t-j)+i$ in $G$.
\end{claim}
\begin{proof}
Since $\varphi$ is strongly rainbow on $P_i$, $P_i^\varphi$ is a linear path of length $i$ in $F^*$. Since $\chi$ is strongly rainbow on $T_S\subseteq T$, $Q_{v_a}\cup Q_{v_{b_i}}$ is  a linear path of length $2(t-1-j)$ in $\bigcup_{j<t} H_j$.  In particular, $V(R_1)\cap V(R_2)=\emptyset$.

By \eqref{eq:at-most-one}, $e_a$ intersects $P_i^\varphi$ only at $a$ and $e_{b_i}$ intersects $P_i^\varphi$
only at $b_i$. Since $e_a, e_{b_i}\in E(H_t)$, $e_a$ intersects $Q_{v_a}^\chi\cup Q_{v_{b_i}}^\chi$ only at $v_a$
and $e_{b_i}$ intersects $Q_{v_a}^\chi\cup Q_{v_{b_i}}^\chi$ only at $v_{b_i}$. Also, $e_a$ and $e_{b_i}$ are disjoint since $e_a\setminus \{v_a\}, e_{b_i}\setminus \{v_{b_i}\}$ are two different edges of $M_t$ and $v_a\neq v_{b_i}$. Hence, $R_i:=P_i^\varphi\cup Q_{v_a}^\chi\cup Q_{v_{b_i}}^\chi \cup \{e_a, e_{b_i}\}$ is a linear cycle of length $2(t-j)+i$ in $G$.
\end{proof}

\medskip

{\bf Case 2.} $|N|\geq |M|$.

\medskip

In this case, we apply Lemma \ref{lem:path-JMY} with $E_1=M$, $E_2=N$, $\ell=2k$.
There exists a strongly rainbow path $a'ab_1b_2\dots b_{2k-1}$ of length $2k$ in $B'$ such that
the first edge is in $M$ and all the other edges are in $N$. Without loss of generality, suppose
$f(v_{a'})=f(v_a)=1$, then we since $b_ib_{i+1}\in N$ for each $i=1,\dots, 2k-2$ we have $f(v_{b_1})=f(v_{b_3})=\cdots=f(v_{b_{2k-1}})=2$. By the same reasoning as  in Case 1,
for each $i\in [k]$, we can use the strongly rainbow path $ab_1\dots b_{2i-1}$, which has length $2i-1$ to find a linear cycle of length $2(t-j)+(2i-1)$ in $G$. These give us
linear cycles in $G$ of lengths $2m+1,2m+3,\dots, 2m+2k-1$.
Next, for each $i\in [k]$, we can use the strongly rainbow path $a'ab_1\dots b_{2i-1}$ to build a linear cycle of length $2(t-j)+2i$ in $G$. These give us linear cycles in $G$ of
length $2m+2,\dots, 2m+2k$. Together, these two collections
give us linear cycles of length $2m+1, 2m+2,\dots, 2m+2k$,
where $m=t-j\leq t$. So, in this case, the claim also holds.
\end{proof}


\subsection{Linear cycles of even consecutive lengths in linear $r$-graphs}

Now we develop our main result for the section. Our result is that for each $r\geq 3$ there are constants $c_1,c_2$, depending only on $r$ such that in every $n$-vertex linear $r$-graph $G$
with average degree $d(G) \geq c_1 k$ we can find
linear cycles of lengths $2\ell+2, 2\ell+4,\dots, 2\ell+2k$ for some $\ell\leq \lceil \frac{\log n}{\log (d(G)/k)- c_2}\rceil-1 $. This would also
immediately yield an improved bound on the linear Tur\'an number of an $r$-uniform linear $2k$-cycle.

\begin{definition}
{\rm Given a positive real $d$, an $r$-graph $G$ is  said to be {\it $d$-minimal}, if $d(G)\geq d$
but for every proper induced subgraph $H$ we have $d(H)<d(H)$.
}
\end{definition}

\begin{lemma} \label{lem:minimal}
Let $d$ be any positive real. If $G$ is an $r$-graph satisfying that $d(G)\geq d$ then
$G$ contains a $d$-minimal subgraph $G'$.
\end{lemma}
\begin{proof}
Among all induced subgraphs $H$ of $G$ satisfying $d(H)\geq d$, let $G'$ be one that minimizes $|V(G')|$. Then $G'$ is $d$-minimal.
\end{proof}

\begin{lemma} \label{lem:boundary}
Let $r\geq 3$ be an integer and $d$ a positive real. Let $G$ be a $d$-minimal $r$-graph. For any proper subset $S$ of $V(G)$, the number of edges of $G$ that contains a vertex in $S$ is at least $d|S|/r$.
\end{lemma}
\begin{proof}
Otherwise, suppose there is a proper subset $S$ of $V(G)$ such that the number of edges of $G$ that contain a vertex in $S$ is at most $d|S|/r$. Then the subgraph $G'$ of $G$ induced by $V(G)\setminus S$ satisfies
\[e(G')\geq e(G)-d|S|/r\geq d|V(G)|/r -d|S|/r=d(|V(G')|/r.\]
Hence $d(G')\geq d$, contradicting $G$ being $d$-minimal.
\end{proof}

\begin{theorem} \label{thm:even-lengths2}
Let $k,r$ be integers where $k\geq 1$ and $r\geq 3$. 
Let $c_3=128 r^{r+3}$ and $c_4=\log(64kr^{r+2})$. 
If $G$ be is an $r$-partite linear $r$-graph 
with average degree $d(G)\geq c_3 k$ then
$G$ contains linear cycles of lengths $2\ell+2, 2\ell+4,\dots, 2\ell+2k$, for some positive integer
$\ell\leq \lceil  \frac{\log n}{\log (d(G)/k)- c_4}\rceil-1$.
\end{theorem}
\begin{proof}
Let $d=d(G)$. By Lemma \ref{lem:minimal}, $G$ contains a $d$-minimal
subgraph $G'$.
Suppose $G'$ does not contain a collection of linear cycles of length $2\ell+2,2\ell+4,\dots, 2\ell+2k$, where $\ell\leq \lceil  \frac{\log n}{\log (d(G)/k)- c_4}\rceil-1$.
We derive a contradiction.
Let us apply Algorithm \ref{maximal-tree} to $G'$ from $x$ and
let $(H,T,\chi)$ be the triple produced. Let $m$ denote the height of $H$ and $T$.

For each $i\in [m]$,
let
\[G_i=\{e\in E(G')\setminus E(H): e\cap V(H_i)\neq \emptyset, e\cap \bigcup_{j<i} V(H_j)=\emptyset\}.\]
Let
\[G^1_i=\{e\in E(G_i): |e\cap V(H_i)|=1\}, \quad \mbox{ and } \quad F_i=\{e\in E(G_i): |e\cap V(H_i)|\geq 2\}.\]

Note that $G_m^1=\emptyset$, as otherwise 
Algorithm \ref{maximal-tree} would have produced non-empty $L_{m+1}$, instead of stopping at step $m$, $L_m$ being the last level.

For convenience, let $p=\lceil  \frac{\log n}{\log (d(G)/k)- c_4)}\rceil$. For convenience, define $L_{m+1}=\emptyset$.

\begin{claim} \label{even-lengths.1.}
For each $1\leq i \leq \min\{m, p\}-1$, 
we have $e(G_i^1)\leq 8kr^3 (|L_i|+|L_{i+1}|)$.
\end{claim}
\begin{proof}
Let $D_i$ be the set of edges of edges in $G_i^1$ that intersect $V(H_i)$ in $L_i$. By
Algorithm \ref{maximal-tree},
\[e(D_i)\geq (1/r) e(G_i^1).\] Let $e\in D_i$. By definition,
$e$ intersects $V(H_i)$ in exactly one vertex and that vertex lies in $L_i$. Furthermore, $e$
contains no vertex in $\bigcup_{j<i} V(H_j)$. If $e\setminus L_i$ is vertex disjoint from $V(H_{i+1})\setminus L_i$,
then $e$ would have been added to $H_{i+1}$ by Algorithm \ref{maximal-tree}, contradicting $e\notin E(H)$.
Hence $e$ must contain at least one vertex in $V(H_{i+1}\setminus L_i)$. If $e(D_i)\geq 8kr(r-1)(|L_i|+|L_{i+1}|)$
then by Lemma \ref{lem:tree+dense1} (with $t=i+1$) $G$ contains linear cycles of lengths $2\ell+2,2\ell+4,
\dots, 2\ell+2k$ for some $\ell \leq i\leq 
\lceil  \min\{m, \frac{\log n}{\log (d(G)/k)- c_4)}\rceil\}-1$, contradicting our assumption. Hence,
\[e(D_i)\leq 8kr(r-1)(|L_i|+|L_{i+1}|) < 8kr^2(|L_i|+|L_{i+1}|).\]
Therefore
\[e(G^1_i)\leq 8kr^3 (|L_i|+|L_{i+1}|).\]
\end{proof}

\begin{claim} \label{even-lengths.2.}
For each $1\leq i\leq \min\{m, p-1\}$ we have $e(F_i)\leq 8kr^{r+2}|L_i|$.
\end{claim}
\begin{proof}
Suppose $e(F_i)\geq 8kr^{r+2}|L_i|$. Then by Lemma \ref{lem:tree+dense2} (with $t=i$), we can find in $G$ linear cycles of
length $2\ell+2, 2\ell+4,\dots, 2\ell+2k$ for some $\ell\leq i \leq
\lceil  \frac{\log n}{\log (d(G)/k)- c_4)}\rceil-1$, contradicting our assumption.
Hence
\[e(F_i)\leq 8kr^{r+2}|L_i|.\]
\end{proof}

By Claims \ref{even-lengths.1.} and \ref{even-lengths.2.}, and noting that $E(G_m^i)=\emptyset$ we have
 
 \begin{equation} \label{Gi-upperbounds}
 \forall 1\leq  i\leq \min\{m, p-1\}\quad e(G_i)=e(G_i^1)+e(F_i)\leq 16kr^{r+2}(|L_i|+|L_{i+1}|). 
\end{equation}

\medskip

\begin{claim} \label{even-lengths.3.}
For each $1\leq i\in \min\{m-1, p-1\}$, $e(\bigcup_{j=1}^i G_i)\geq (d/2)\sum_{j=0}^i |L_j|$.
\end{claim}

\begin{proof}  
Let $S=\bigcup_{j=0}^i V(H_i)$. Since $i\leq m-1$, $S$ is a proper
subset of $V(G')$. Let $E_S$ denote the set of edges of $G'$ that
contains a vertex in $S$. By our definitions, $E_S\subseteq  \cup_{j=1}^iE(H_j)\cup \bigcup_{j=1}^i G_j$.
Since $G'$ is $d$-minimal, by Lemma \ref{lem:boundary},
\[|E_S|\geq d|S|/r = d(1+\sum_{j=1}^i(r-1)|L_j|)/r.\]
On the other hand, by the definition of $H$, $|\bigcup_{j=1}^i E(H_j)|=\sum_{j=1}^i |L_j|$.
Hence
\[e(\bigcup_{j=1}^i G_i)=|E_S|- |\bigcup_{j=1}^i E(H_j)| \geq d(1+\sum_{j=1}^i(r-1)|L_j|)/r -\sum_{j=1}^i |L_j|\geq \sum_{j=1}^i |L_j|(d(1-1/r)-1)+d/r\geq (d/2) \sum_{j=1}^i |L_j|.\]
\end{proof}

By \eqref{Gi-upperbounds}, we have
\begin{equation}\label{Gi-union-upper}
\forall 1\leq i\leq \min\{m-1, p-1\} \quad e(\bigcup_{j=1}^i G_i)\leq \sum_{j=1}^i 16kr^{r+2} (|L_j|+|L_{j+1}|).
\end{equation}

For each $i=0,\dots, m$, let $U_i=\bigcup_{j=0}^i L_i$.
By \eqref{Gi-union-upper}, $\forall 0\leq i \leq \min\{m-1,p-1\}$

$$  32kr^{r+2}|U_{i+1}| \geq \sum_{j=1}^i 16kr^{r+2} (|L_j|+|L_{j+1}|)\geq 
(d/2) \sum_{j=1}^i |L_j|\geq (d/2) |U_i|.$$

Hence,
\begin{equation} \label{U-iteration}
\forall 0\leq i\leq \min\{m-1, p-1\} \quad |U_{i+1}|\geq (d/64kr^{r+2}) |U_i|.
\end{equation}

\begin{claim} \label{claim:mp} $m\geq p$.
\end{claim}
\begin{proof}
Suppose otherwise. Let $S=V(H_m)\setminus (L_{m-1}\cup L_m)$.
Then $S$ is a proper subset of $V(G')$ with  $|S|=(r-2)|L_m|$.
Let $E_S$ denote the set of edges of $G'$ that contain a vertex in $S$.
Since $G'$ is $d$-minimal, we have
\[|E_S|\geq d|S|/r=d|L_m|(r-2)/r.\]
On the other hand,  since $L_m$ is the last level of $H$,
by the definitions, $E_S\subseteq E(H_m)\cup \bigcup_{i=1}^{m-1} E(G_i)\cup F_m$.
By \eqref{Gi-union-upper}, Claim \ref{even-lengths.2.} and the fact that
 $e(H_m)=|L_m|$,  we have
 \begin{eqnarray*} \label{ES-upper}
 |E_S|&\leq& |L_m|+ \sum_{j=1}^{m-1} 16kr^{r+2} (|L_j|+|L_{j+1}|)+ 8kr^{r+2}|L_m|\\
 &\leq& 32kr^{r+2} |U_{m-1}|+16kr^{r+2} |L_m|.
 \end{eqnarray*}
 
Combining the lower and upper bounds above on $|E_S|$, we get

\[ (r-2)d/r|L_m|\leq 32kr^{r+2} |U_{m-1}|+16kr^{r+2} |L_m|.\]

Since $d\geq c_3k=128 r^{r+3}k$. We have
$d/r\geq 128kr^{r+2}$. This inequality above implies $|L_m|<|U_{m-1}|$ and
thus $|U_m|=|U_{m-1}|+|L_m|\leq 2|U_{m-1}|$. But by \eqref{U-iteration}, we have
\[|U_m|\geq (d/64kr^{r+2})|U_{m-1}\geq 2|U_{m-1}|,\] a contradiction.
\end{proof}

By Claim \ref{claim:mp} $m\geq p$. But now we show that this would mean the expansion rate was so fast that $|U_p|>n$, a contradiction.  Recall that $|U_0|=|L_0|=1$. Thus by \eqref{U-iteration} 
\[|U_p|\geq  (d/64kr^{r+2})^p.\]

Taking logarithm both sides of the inequality and using
that $c_4=\log 64kr^{r+2}$, we get 

\begin{eqnarray*}
\log |U_p|&\geq& p [\log (d/k)-\log (64kr^{r+2})]\\
&\geq & \frac{\log n}{ \log (d/k)-c_4}   \left [\log (d/k)-\log (64kr^{r+2})\right]\\
&=& \log n,
\end{eqnarray*}
which gives $|U_p|>n$, a contradiction. This completes the proof of the theorem.

 \end{proof}

Finally we are ready to prove Theorem \ref{thm:even-lengths}. We need the following
result of Erd\H{o}s and Kleitman.
\begin{lemma}\cite{EK}\label{lem:EK}
Let $r\geq 2$. Every $r$-graph $G$ contains an
$r$-partite subgraph $G'$ with $e(G')\geq (r!/r^r) e(G)$.
\end{lemma}

{\bf Proof of Theorem \ref{thm:even-lengths}:}
Let $r\geq 3, k\geq 2$ be the given integers.
Let $c_3,c_4$ be the constants obtained in Theorem \ref{thm:even-lengths2}.  Let $c_1=c_3 r^r=128r^{2r+3}$ and $c_2=c_4+\log (r^r)=\log (64kr^{2r+2})$.
Let $G$ be an $n$-vertex $r$-graph with $d(G)\geq c_1k$. By Lemma \ref{lem:EK}, $G$ contains an $r$-partite subgraph $G'$ with $d(G')\geq d(G) (r!/r^r)\geq d(G)/r^r\geq c_3k$. By Theorem \ref{thm:even-lengths2}, $G'$ (and thus $G$ als0) contains linear cycles of lengths $2\ell+2, 2\ell+4,\dots, 2\ell+2k$, for some positive integer
\[\ell\leq \left\lceil  \frac{\log n}{\log (d(G')/k)- c_4)}\right \rceil-1\leq \left \lceil\frac{\log n}{\log (d(G)/k)-\log r^r- c_4)}\right\rceil-1 \leq 
\left \lceil\frac{\log n}{\log (d(G)/k)- c_2)}\right\rceil-1. 
\]. 
\qed
\medskip 

As mentioned in the introduction, as a quick application of Theorem \ref{thm:even-lengths}, we obtain an improvement (in Corollary \ref{cor:linear-turan}) on the
bound given in \cite{CGJ} on the linear Tur\'an number of an even cycle by reducing the 
coefficient from at least exponential in $k$ to a function quadratic in $k$ (for fixed $r$).

\medskip

{\bf Proof of Corollary \ref{cor:linear-turan}:}
Let $r\geq 3, k\geq 2$ be the given integers.
Let $c_1=128r^{2r+3}$ and $c_2=\log (64kr^{2r+2})$, as in Theorem \ref{thm:even-lengths}.
Let $c_3=64kr^{2r+3}$.
Let $G$ be an $n$-vertex $r$-graph with $e(G)\geq c_3 k n^{1+1/k}$.  Then 
$d(G)\geq c_3 r kn^{1/k} \geq c_1k$ thus we can apply  Theorem \ref{thm:even-lengths} to $G$ and obtain that
it contains linear cycles of lengths $2\ell, 2\ell+4,\dots, 2\ell+2k-2$ for some
\[\ell\leq \left\lceil  \frac{\log n}{\log (d(G)/k)- c_2)} \right \rceil \leq  \left\lceil \frac{\log n}{\log (c_3 r) +\log{n^{1/k}} - c_2}
\right\rceil  \leq k. \]
Therefore the even numbers in the interval $[2\ell,\ldots, 2\ell+2(k-2)]$ contain the number $2k$, which means $G'$ must contain a linear cycle of length exactly $2k$.
\qed


\section{Concluding remarks} \label{section:conclusion}

We do not know if we can improve the bound on the shortest lengths of the cycles guaranteed
in Theorem \ref{thm:all-lengths} to a similar one as in Theorem \ref{thm:even-lengths}. 

\begin{question} \label{question:all-lengths}
Let $r\geq 3$ and $k\geq 1$ be integers. Is it true that there exist constants $c_1=c(r), c_2=c(r,k)$
such that if $G$ is an $n$-vertex linear $r$-graph with average degree $d(G)\geq c_1k$ then
$G$ contains linear cycles of $k$ consecutive lengths, the shortest of which is at most $2\lceil \frac{\log n}{\log d(G)/k-c_2}\rceil$?
\end{question}

A weaker question is the following analogue for odd linear cycles.

\begin{question} \label{question:odd-lengths}
Let $r\geq 3$ and $k\geq 1$ be integers. Is it true that there exist constants $c_1=c(r), c_2=c(r,k)$
such that if $G$ is an $n$-vertex linear $r$-graph with average degree $d(G)\geq c_1k$ then
$G$ contains linear cycles of $k$ consecutive odd lengths, the shortest of which is at most $2\lceil \frac{\log n}{\log (d(G)/k)-c_2}\rceil$?
\end{question}

If the answer to Question \ref{question:odd-lengths} is affirmative, then it would give better bounds on the known upper bounds on $ex_L(n,C_{2k+1})\leq cn^{1+1/k}$, reducing the coefficient $c$ from being exponential in $k$ to being quadratic in $k$, just like in the Corollary~\ref{cor:linear-turan}.
As a good starting point to address Questions \ref{question:all-lengths} and \ref{question:odd-lengths} consider the case $k=1$.

We would like to mention the following result of Ergemlidze, Gy\H{o}ri and Methuku \cite{EGM}.
\begin{theorem} {\rm (Theorem 3 in \cite{EGM})} \label{thm:EGM}
Let $\cC^r_m$ denote the family of $r$-uniform linear cycles of length at most $m$.
If $\ex(n,\cC^2_{2k-2})\geq c n^\alpha$ for some $c,\alpha>0$ then
$\ex_L(n, \cC^3_{2k+1})= \Omega(n^{2-\frac{1}{\alpha}})$.
\end{theorem}

\noindent A famous conjecture in extremal graph theory, due to Erd\H{o}s and Simonovits \cite{Erd65, ES-compact} asserts that $\ex(n,\cC^2_{2k})=\Omega(n^{1+1/k})$ for any $k\geq 2$.
This is only known to be true for $k\in \{2,3,5\}$ (see \cite{degenerate-survey} for further details).
Hence, Theorem \ref{thm:EGM} yields the following.

\begin{corollary} {\rm (\cite{EGM})} \label{cor:EGM}
For any $k\in \{2,3,4,6\}$, $\ex_L(n,\cC^3_{2k+1})= \Omega(n^{1+1/k})$.
\end{corollary}

\noindent Interestingly the proof of the above result of Ergemlidze, Gy\H{o}ri and Methuku \cite{EGM} in fact also works in the sparse range.
The following property can be easily derived from the construction of \cite{EGM} (see its final section):
If there exists a $\cC^2_{2k-2}$-free graph $G$ with average degree $d$, then there exists a $\cC^3_{2k+1}$-free $3$-graph on $\Theta(e(G))$ vertices and with average degree at least $\Omega(d)$.

By Corollary \ref{cor:EGM}, the bounds on the shortest lengths of the cycles
in Questions \ref{question:all-lengths} and \ref{question:odd-lengths}, if
true, are best possible when $r=3$ and $k=2,3,4,6$.
If the above-mentioned Erd\H{o}s-Simonovits conjecture on $\ex(n,\cC^2_{2k})$ is true then the bounds in these questions would be optimal for $r=3$ and for all $k$.



\section{Appendix}

{\bf Proof of Proposition \ref{prop:optimal-lengths}:}
A partial $(n,k,q)$-{\it Steiner system} is a family $\cF$ of $k$-subsets on $[n]$ such that every $q$-subset of $[n]$ is in at most one member of $\cF$. In particular, a partial $(n,k,2)$ Steiner system is a linear hypergraph.
R\"odl \cite{rodl} showed that for all fixed $k>q\geq 2$, as $n\to \infty$ there exist partial $(n,k,q)$-Steiner systems of size $(1-o(1))\binom{n}{q}/\binom{k}{q}$ (see \cite{keevash}, \cite{glock} for recent breakthroughs on the existence of steiner systems).
Let $m=\lfloor (1-\epsilon) \log_d n\rfloor$. By our discussion above, we can find  a large enough integer $n_0$ such that for all $n\geq n_0$
there exists an $n$-vertex partial $(n,r,2)$-Steiner system $G$ of size at least $0.9 \binom{n}{2}/\binom{r}{2}$ and that the following 
inequality also holds

\begin{equation} \label{n0-condition}
0.8d n^{\epsilon^2} > 2^{m+1} r.
\end{equation}
By definition, $G$ is a linear $r$-graph. Set $p=2rd/n$ and let $F$ be a random subgraph of $G$ obtained by independently including each edge of $G$ with probability $p$.
Let $\mathbb{X}$ denote the number of edges in $F$ and $\mathbb{Y}$ the number of linear cycles of length at most $m$ in $F$. Then

\[\mathbb{E}[\mathbb{X}]\geq
0.9\binom{n}{2}/\binom{r}{2} \cdot (2rd/n)>1.8dn/r.\]

On the other hand, observe that for any fixed $\ell$, there are fewer than $n^\ell$ ways to choose a cyclic list $v_1v_2\dots v_\ell v_1$.
Since $G$ is linear, for each cyclic list $v_1v_2\dots v_\ell v_1$ there is at most one linear cycle in $G$ with $v_1v_2\dots v_\ell v_1$ being its skeleton. So there are fewer than $n^\ell$ linear cycles of length $\ell$ in $G$. Hence, using $d\geq (2r)^{\frac{1}{\epsilon^2}}$ and $m\leq (1-\epsilon) \log_d n$, we have

\[\mathbb{E}[\mathbb{Y}]\leq \sum_{\ell=3}^m n^\ell (2rd/n)^\ell =\sum_{\ell=3}^m (2rd)^\ell<2(2rd)^m<2^{m+1} d^{(1+\epsilon) m}\leq 2^{m+1} n^{1-\epsilon^2}\]

Therefore,  by \eqref{n0-condition},

\[\mathbb{E}[\mathbb{X}-\mathbb{Y}] > \frac{1.8dn}{r}-2^{m+1} n^{1-\epsilon^2}>\left(\frac{1.8d}{r}-\frac{2^{m+1}}{n^{\epsilon^2}} \right)n \geq \frac{dn}{r}.\]

Hence there exists an $F$ for which $\mathbb{X}-\mathbb{Y}\geq \frac{dn}{r}$.
From $F$ let us delete one edge from each linear cycle of length at most $m$. Let $H$ be the remaining graph.
Then $H$ is an $n$-vertex linear $r$-graph that has average degree at least $d$ and has no linear cycles of length at most $(1-\epsilon)\log_d n$.
\qed
\end{document}